\documentclass[11pt]{amsart}
\usepackage{amssymb}

\usepackage[T1]{fontenc} 
\usepackage[utf8]{inputenc} 

\usepackage[unicode,colorlinks,
	citecolor=blue,
	urlcolor=blue,
	linkcolor=blue
]{hyperref} 

\usepackage{soul} 

\usepackage{url}

\theoremstyle{plain}
\newtheorem{thm}{Theorem}[section]
\newtheorem{prop}[thm]{Proposition}
\newtheorem{lem}[thm]{Lemma}

\theoremstyle{definition}
\newtheorem{defn}[thm]{Definition}

\theoremstyle{remark}
\newtheorem{remark}[thm]{Remark}

\numberwithin{equation}{section}

\renewcommand{\epsilon}{\varepsilon}
\renewcommand{\phi}{\varphi}

\renewcommand{\preceq}{\preccurlyeq}
\renewcommand{\succeq}{\succcurlyeq}

\newcommand{\N}{\mathbb{N}}

\newcommand{\de}{\delta}

\newcommand{\lb}{\lbrace}
\newcommand{\lk}{\lbrack}
\newcommand{\rb}{\rbrace}
\newcommand{\rk}{\rbrack}

\newcommand{\ra}{\rightarrow}

\newcommand{\da}{\downarrow}

\begin{document}

\title{There is a Deep 1-Generic Set}

\author{Ang Li}

\date{\today}

\begin{abstract}
An infinite binary sequence is Bennett deep if, for any computable time bound, the difference between the time-bounded prefix-free Kolmogorov complexity and the prefix-free Kolmogorov complexity of its initial segments is eventually unbounded. It is known that weakly 2-generic sets are shallow, i.e.\ not deep. In this paper, we show that there is a deep $1$-generic set. 
\end{abstract}

\maketitle

\section{Introduction}
In \cite{Deep}, Bennett defined logical \emph{depth} to capture the difference between sets containing a lot of ``useful information'' that cannot be acquired in a short time and sets with noise and useless information. For example, a math paper is deep in the sense that the theorems in it can be produced from the basic definitions but reproducing these theorems might take a long time. Given an initial segment of a binary string corresponding to a deep set, a compressor should not be able to give us the complexity of the string within computable time bounds---the compression becomes better when more time is allowed. 
\par
Prefix-free Kolmogorov complexity and its resource-bounded version were used by Bennett to define depth. A \emph{prefix-free} machine $M$ is a partial function from $\lb0,1\rb^{\ast}$ to $\lb0,1\rb^{\ast}$ such that no element in the domain is the prefix of any other element in the domain.
\begin{defn}
The \emph{prefix-free Kolmogorov complexity} $K_M(x)$ of a binary string $x$ with respect to a prefix-free machine $M$ is
    $$K_M(x):=\min\lb|\sigma|:M(\sigma)=x\rb,$$ where $|\sigma|$ is the length of $\sigma$.
\end{defn}
There exists optimal prefix-free machines in the following sense:
\begin{defn}
We say that a machine $R$ is an \emph{optimal} prefix-free machine if $R$ is prefix-free, and for each prefix-free machine $M$ there is a constant $d_M$ such that $$(\forall x)
\lk K_R(x)\leq K_M(x)+d_M\rk.$$ The constant $d_M$ is called the coding constant of $M$ (with respect to $R$).
\end{defn}
From now on, we fix an optimal machine $\mathbb{U}$ by setting $\mathbb{U}(0^{e-1}1\sigma)=M_e(\sigma)$ for any $e,\sigma$, where $\lb M_e\rb_{e\in\N}$ is defined in \ref{g0}. For any natural number $s$, $K_s(x) = \min\lb|\sigma|:\mathbb{U}(\sigma) = x\mathrm{\ in\ at\ most\ }s\mathrm{\ many\ steps}\rb$. See Downey and Hirschfeldt \cite{downey} and Nies \cite{Nies} for more information about Kolmogorov complexity.
\begin{defn}
For a computable function $t\colon\mathbb{N}\ra\mathbb{N}$ and a prefix-free machine $M$, the \emph{prefix-free Kolmogorov complexity with time bound $t$} relative to $M$ is $$K_M^t(x):=\min\lb|\sigma|:M(\sigma)\da=x\mathrm{\ in\ at\ most\ }t(|x|)\mathrm{\ many\ steps}\rb,$$
and we write $K^t$ for $K^t_{\mathbb{U}}$.
\end{defn}
Now we present the precise definition of Bennett's depth.
\begin{defn}
For $X\in 2^{\N}$, we say that $X$ is \emph{deep} if for every computable time bound $t$ and $c\in\N$, $$(\forall^{\infty}n)\lk K^t(X\upharpoonright n)-K(X\upharpoonright n)\geq c\rk.$$
A set that is not deep is called \emph{shallow}.
\end{defn}
The class of shallow sets is comeager in the Cantor space as we see later in Proposition \ref{comeager}. Therefore, it is reasonable to ask if generic sets are shallow. However, we shall show that there exists a deep 1-generic set in this paper. For a set $S$ of finite binary strings, a set $A$ \emph{meets} $S$ when there is an $n$ such that $A\upharpoonright n\in S$, and $A$ \emph{avoids} $S$ when there is an $n$ such that $\sigma\not\in S$ for all $\sigma$ extending $A\upharpoonright n$. A set $A$ is
\emph{$n$-generic} if it meets or avoids every $\Sigma^0_n$ set of strings. A set $A$ is \emph{weakly $n$-generic} if it meets all dense $\Sigma^0_n$ sets of strings. Every weakly ($n+1$)-generic set is
$n$-generic. For more about genericity, see Downey and Hirschfeldt \cite{downey}.
\begin{thm}
There exists a deep 1-generic set.
\end{thm}
In \cite{TimeBound}, H\"{o}lzl, Kr\"{a}ling, and Merkle presented a method to show that there is a deep set of any high degree. Intuitively, computationally ``easy enough'' sets and sufficiently random sets cannot be deep. Indeed, computable and Martin-L\"of random sets are shown to be shallow by Bennett \cite{Deep}. But there is a computably enumerable deep set, i.e., the halting set. To explore the boundary of ``easiness'', Downey, McInerney, and Ng \cite{LowDeep} proved that there exists a superlow c.e.\ deep set. Their method is a variation on that of H\"{o}lzl et al.
\par
We modify this method to show that there is a deep 1-generic set. Unlike the previous proofs, we use a $\emptyset''$-priority construction to build a deep set while leaving room to meet the genericity requirements in our proof. See \cite{Tree} for an introduction to $\emptyset''$-priority constructions.
\par
Before we present the proof of the theorem, we show that our result is optimal in terms of genericity. This is first proved by Bienvenu, Delle Rose, and Merkle (unpublished). We provide a proof for reference. In the proof, we make use of a Solovay function.

\begin{defn}
    A function $f\colon\N\ra\N$ is a \emph{Solovay function} if $K(n)\leq^+ f(n)$ for all $n$ and $K(n)=^+f(n)$ infinitely often. Here, $a_0\leq^+a_1$ means $(\exists c\in\N)(\forall n)\lk a_0(n)\leq a_1(n)+c\rk$ for any functions $a_1,a_2$.
\end{defn}
\begin{prop}\label{comeager}
Every weakly 2-generic set is shallow.
\end{prop}
\begin{proof}
Define a computable Solovay function $h$ as the following:
\begin{equation}
    h(\langle n,s\rangle)=
        \begin{cases}
            K_s(n) & \text{if $K_s(n)\not=K_{s-1}(n),$}\\
            +\infty & \text{otherwise}
        \end{cases}
\end{equation}
(here we use $+\infty$ for convenience; any coarse upper bound of $K(n)$ would do).
\par
Now we build $V\subseteq\lb0,1\rb^{\ast}$ using $\emptyset'$ such that 
\begin{equation}\label{1.2}
    (\forall\sigma)(\exists\tau\succeq\sigma)\lk\tau\in V\wedge K^t(\tau)\leq^+K(|\tau|)\rk,
\end{equation}
where $t$ is a suitable computable time bound.
\par 
Suppose $\sigma_n$ is the $n$th string in $\lb0,1\rb^*$. Using $\emptyset'$, we can compute the number of steps $s_n$ such that $K_{s_n}(n)=K(n)\not=K_{s_n-1}(n)$. Let $\tau_n=\sigma_n0^{\langle n,s_n\rangle-|\sigma_n|}$. A reasonable choice of the pairing function $\langle\cdot,\cdot\rangle$ ensures that $\tau_n$'s are all well-defined. There is a machine $M$ that outputs $\sigma_n0^{\langle n,s\rangle-|\sigma_n|}$ given input $\delta$ where $n=\mathbb{U}(\de)$ and $\mathbb{U}$'s computation takes exactly $s$ steps. We enumerate $\tau_n$ into $V$.
\par
Now we verify that $V$ satisfies \ref{1.2}. Notice that $|\tau_n|=\langle n,s_n\rangle$. By Remark \ref{g0} below, there exists a computable function $t$ such that $K^t(\tau_n)$, the complexity of $\tau_n$ within this time bound, is smaller than the complexity $K_M(\tau_n)$ within $\langle n,s_n\rangle$ many steps by a constant that does not depend on $n$. But $K_M(\tau_n)$ is no larger than $K_{s_n}(n)=h(\langle n,s_n\rangle)$. Therefore, the following holds:
\begin{equation}
    K^t(\tau_n)\leq^+h(\langle n,s_n\rangle)=K(n)\leq^+K(\langle n,s_n\rangle)=K(|\tau_n|)\leq^+ K(\tau_n).
\end{equation}
The second inequality holds because the existence of computable functions from $\langle a,b\rangle$ to each of its coordinates. 
\par Using $V$, we show that every weakly 2-generic set is shallow. Suppose $A$ is weakly 2-generic. Let $\lb V_n\rb_{n\in\N}$ be the subsets of $V$ such that $V_n$ only contains strings of length at least $n$. Each $V_n$ is a dense $\Delta^0_2$ set. Hence, $A$ meets each $V_n$, i.e. for each $n$, $A$ has an initial segment $\tau$ of length at least $n$ that is in $V_n$ and $K^t(\tau)\leq^+K(\tau)$. Therefore, there exists $t$ and $c$ such that $K^t(A\upharpoonright n)\leq K(A\upharpoonright n)+c$ infinitely often. So, $A$ is shallow.
\end{proof}
\begin{remark}\label{g0}
For any computable $t$ and prefix-free machine $M$, there exists a computable function $g_M$ and some constant $c$ such that $K^{g_M(t)}(x)\leq K^{t}_M(x)+c$. To see this, we utilize an algorithm from \cite[Definition 3.1.2]{LiVitanyi}. For any partial computable function $\varphi_e$, we define a machine $M_e$ that does the following when the input is $\sigma\in\lb0,1\rb^{\ast}$:
\begin{enumerate}
    \item Let $\tau_0$ be the empty string $\epsilon$. Let $\gamma_i$ be the $i$th string of $\lb0,1\rb^{\ast}$.
    \item In stage $s$, we consider substages. In substage $0$, $M_e$ computes $\varphi_{e,0}(\tau_s\gamma_0)$. $M_e$ does one more stage of computation of each $\varphi_e(\tau_s\gamma_i)$ for all $i\leq t$ in substage $t$. If $\varphi_e(\tau_s\gamma_{i_0})$ is the first to converge for some $i_0$, we go to step 3.
    \item If $\gamma_{i_0}$ is $\epsilon$, $M_e$ outputs $\varphi_e(\tau_s)$. Otherwise, read one more bit $b$ of $\sigma$ and let $\tau_{s+1}$ be $\tau_s^{\smallfrown}b$. Go to step 2.
\end{enumerate}
Now, given any prefix-free machine $M$, let $\phi_e$ be its corresponding prefix-free function. Then, $M_e$ is the self-delimiting machine obtained by the algorithm above. Suppose $M$ converges on some input $\sigma$ in $s$ steps. $M_e$ can simulate it in no more than $\sum\limits_{i=1}^{|\sigma|+1}\sum\limits_{j=1}^{2^i+s-2}j=O(4^{|\sigma|}+s2^{|\sigma|}+s^2(|\sigma|+1))$ steps of $\phi_e$ computations. Let $t$ be the time bound such that $M(\sigma)=x$ halts in $s=t(|\sigma|)$ steps. Notice that there is an overhead that is computable in $(s,|\sigma|)$ for $M_e$ to run the algorithm depending on the underlying machine model. If we assume that $s\geq|\sigma|$ and take the overhead into consideration, $\mathbb{U}$ can simulate it in $g_M(t(|\sigma|))$ steps for some computable function $g_M$ that depends on the index $e$ of $M$.
\par
The time bound of the simulation can be efficient if we restrict ourselves to self-delimiting machines as in \cite[Example 7.1.1]{LiVitanyi}. Specifically, it can be $et\log t$, where $e$ is the coding constant. However, the problem of whether polynomial time is enough given any prefix-free machine is still open. For the approach using the efficient simulation of self-delimiting machines, Juedes and Lutz \cite{poly} showed the negative result that an efficient simulation of any prefix-free function using a self-delimiting machine exists if and only if $\mathrm{P}=\mathrm{NP}$. For more information about the open question, see Li and Vitanyi \cite{LiVitanyi}.
\end{remark}
\section{The Proof of Theorem 1.5}
Let $\lb\varphi_i\rb_{i\in\N}$ be a listing of all partial computable functions. We partition $\N$ into consecutive intervals $I_0,I_1,\ldots$ where the interval $I_j$ has length $2^j$. We assign $\varphi_0$ to every second interval including the first one, $\varphi_1$ to every second interval including the first one of the remaining intervals, and so on for $\varphi_2,\varphi_3,\ldots.$ This way, $\varphi_i$ will be assigned to every $2^{i+1}$th interval. Therefore, if $\varphi_i$ is assigned to $I_j$, then $I_{j+2^{i+1}}$ is the least interval beyond $I_j$ to which $\varphi_i$ is also assigned. In order to make sure that the differences between the time-bounded complexity and the real complexity are large enough, we shall move in some of these intervals by tentatively defining our set $A$ on that interval. The precise definition of ``moving in'' is in the strategies below.
\par
We now construct a deep $1$-generic set $A$. At each stage $s$, we construct $A_s$ so that $A=\lim_{s\ra\infty}A_s$. Let $A_0=\emptyset$. For every $i,e\in\N$, we need to satisfy the following
\vspace{\baselineskip}
\par
\emph{Requirements:} 
\begin{itemize}
    \item [$D_i:$] If $\phi_i$ is an order function, i.e.\ a computable non-decreasing unbounded function from $\N$ to $\N$, then, $(\forall c)(\forall^{\infty}m)\lk K^{\varphi_i}(A\upharpoonright m+1)>K(A\upharpoonright m+1)+c\rk,$
    \item [$G_e:$] $A$ meets or avoids $S_e$, where $\lb S_e\rb_{e\in\N}$ is an effective listing of c.e.\ sets of strings.
\end{itemize}
\par
We order the requirements in the sequence $D_0<G_0<D_1<G_1<\cdot\cdot\cdot$ (of order type $\omega$). Let $\Lambda_0=\lb\infty<\cdot\cdot\cdot<w_n<\cdot\cdot\cdot<w_2<w_1<w_0\rb$ and $\Lambda_1=\lb s<w\rb$, where $s$ means stop and $w$ means wait. Let $\Lambda=\lb\infty<\cdot\cdot\cdot<w_n<\cdot\cdot\cdot<w_2<w_1<w_0<s<w\rb$ be the set of outcomes. Now, we can inductively define the \emph{tree of strategies} $T\subseteq\Lambda^{<\omega}$ by assigning to all strategies $\alpha$ of length $e$ the $e$th requirement in this list and by letting $\alpha^{\smallfrown}o$ be the immediate successors of $\alpha$ where
$o$ ranges over all possible outcomes of $\alpha$. In this tree $T$, we say that a node $\alpha$ is below (resp.\ above) another node $\beta$ if $\alpha\succ\beta$ (resp.\ $\alpha\prec\beta$). In other words, the priority tree is going downward. Also, we say a node $\alpha$ is to the left (resp.\ right) of another node $\beta$ if there are $\gamma\in T$ and $o_{\alpha}<_{\Lambda}o_{\beta}$ (resp.\ $o_{\alpha}>_{\Lambda}o_{\beta}$) such that $\gamma^{\smallfrown}o_{\alpha}\preceq\alpha$ and $\gamma^{\smallfrown}o_{\beta}\preceq\beta$.
\vspace{\baselineskip}
\par
\emph{Strategies}: Define $l_{\epsilon}=0$ where $\epsilon$ is the empty string.
For a $D_i$ node $\alpha$ at stage $s$, we have the following procedures:
\begin{enumerate}
    \item If node $\alpha$ is visited the first time, i.e.\ eligible to act for the first time, we initialize node $\alpha$ by setting a parameter $l_{\alpha}=\max\lb \max \lb l_{\sigma}:\sigma$ has been initialized$\rb,\max \lb l_{\sigma}':\sigma$ has outcome stop$\rb\rb+1$. 
    \item Set $x=0$.
    \item Wait for a number $y$ such that $\varphi_{i,s}(x)=y$.
    \item For the least interval $I_j$ assigned to $\varphi_i$ such that $I_j$ has not been moved in or was marked fresh (defined below), $x\geq\max I_{j+2^{i+1}}+1$, $\varphi_i(\max I_{j+2^{i+1}}+1)\lk s\rk$ converges, and $\min I_j$ is larger than $l_{\alpha}$, we run the universal prefix-free machine on all inputs of length equal to or smaller than $|I_j|-1$ for $g_M(g_N(\varphi_i(\max{I_{j+2^{i+1}}}+1)))$ many steps each, where $g_M,g_N$ are the computable functions as in Remark \ref{g0} and machines $M,N$ are defined in Lemma \ref{lowerbound} below. Then, we choose the leftmost string $\tau$ of length $|I_j|$ which was not among the outputs (there is at least one such string), and alter $A_{s-1}$ so that $A_s\upharpoonright \max I_j+1=(A_{s-1}\upharpoonright\min I_j)^{\smallfrown}\tau$ (we call this action \textquotedblleft moving in $I_j$\textquotedblright). For any moved-in interval $I$ assigned to some $\varphi_{i_0}$ and $\min I>\max I_j$, we mark it fresh. Whether or not we move in an interval, increase $x$ by 1 and go back to step 3.
\end{enumerate}
\par
\emph{True outcomes} of the $D_i$-strategy: Finitary outcomes correspond to the scenarios that the procedure waits at step 3 forever for some $x$ while $\infty$ corresponds to the scenario that $\varphi_i$ is total and step 4 is visited infinitely often.
\par
We let the \emph{current outcome} of the $D_i$-strategy at stage $s$ be $w_{x}$ if the $x$ remains unchanged during this stage. Otherwise, we let the current outcome be $\infty$. In this case, we move from step 3 to step 4 and move back to step 3. So, by the $\Pi_2$-Lemma, a true finitary outcome of a strategy is the current outcome of the strategy at cofinitely many stages whereas a true infinitary outcome of the strategy is the current outcome only at infinitely many stages. For the statement and the proof of the $\Pi_2$-Lemma of our construction, see Lemma \ref{pi2} below.
\vspace{\baselineskip}
\par
For a $G_e$ node $\alpha$: If node $\alpha$ is visited the first time, we initialize node $\alpha$ by assigning a number $c_{\alpha}$ to it and setting a large parameter $l_{\alpha}$. The number $c_{\alpha}$ equals the number of $G_e$ nodes initialized before $\alpha$. The parameter $l_{\alpha}$ is beyond any interval $I_{j+2^{i+1}}$ such that $I_j$ has been moved in and assigned to some $\varphi_i$ with $D_i$ having a finite current outcome and $i\leq e$, and larger than any $l_{\sigma}'$, and larger than $2^{2^{e+1}+2^{c_{\alpha}+1}}$. ($2^{2^{e+1}+2^{c_{\alpha}+1}}$ can be replaced by any large enough strictly increasing computable order function of $(e,c_{\alpha})$; we choose this function to make it easier to show that the real complexity of the initial segment of $A$ is small enough.) Recall that $\lb S_e\rb_{e\in\N}$ is an effective listing of c.e.\ sets of strings. We say that $\alpha$ can be satisfied through $n\geq l_{\alpha}$ at stage $s$ if $A_{s-1}\upharpoonright n$ has an extension $(A_{s-1}\upharpoonright n)^{\smallfrown}\sigma$ in $S_e\lk s\rk$, and to ensure that $G_e$ is not injured on the true path eventually, any interval $I$ assigned to some $\varphi_i$ such that $D_i<G_e$, the current outcome of the $D_i$-strategy is infinite, and $I$ overlaps the extension, has been moved in. We call $\sigma$ the concatenating segment. When there is such an $n$, we want to act by extending the initial segment of $A_{s-1}$ by using the concatenating segment $\sigma$ and declare $G_e$ to be satisfied and then compress the initial segments of $A$ to make sure that the difference between $K^{\varphi_i}(A\upharpoonright m+1)$ and $K(A\upharpoonright m+1)$ is at least $e+c_{\alpha}$ for $i\leq e$ and every $m$ affected by the moved-in $\phi_i$ intervals overlapped by $\sigma$. Therefore, we have to wait for a large enough $n$ such that the weight is small enough. The weight of the strings at time $\varphi_i$ is $$v_{\alpha,i}=\sum_{\theta\in N_{\alpha,i}}2^{-K^{\varphi_i}(\theta)},$$ and $$N_{\alpha,i}:=\lb\theta:\max{I_{j_0}}<|\theta|\leq\max{I_{j_1}}+1\rb,$$ where $I_{j_0}$ is the first interval $\sigma$ overlaps that is assigned to a $\varphi_{i}$ and $I_{j_1}$ is the least interval beyond $\sigma$ that is assigned to $\varphi_{i}$ (notice that we will compress not just the strings that are comparable to $A_{s-1}$ because $A\upharpoonright n$ might change in later stages). In order to compress each of these strings by $e+c_{\alpha}$ bits, we would need to enumerate the request $(K^{\varphi_i}(\theta)-e-c_{\alpha},\theta)$ for every $\theta\in N_{\alpha,i}$. Let $N_{\alpha}=\bigcup_iN_{\alpha,i}$. To make sure that we obtain a bounded request set $R$ with its weight no greater than $1$ eventually, we assign weight to each $N_{\alpha,i}$ set of strings for any $e$ and any $G_e$ node $\alpha$ in advance. For any $G_e$ node $\alpha$, we assign weight $w_{\alpha}$ such that $2^{e+1}\cdot w_{\alpha}= 2^{-e-1-2c_{\alpha}}$. Then, we can assign weight $w_{\alpha,i}$ such that $2^{i+1}\cdot w_{\alpha,i}=2^{-2e-2-2c_{\alpha}}$ to $N_{\alpha,i}$ to compress the strings corresponding to time $\varphi_i$. We want to wait until $v_{\alpha,i}\leq w_{\alpha,i}=2^{-2e-i-2c_{\alpha}-3}$ for every $i\leq e$ to act. However, for sufficiently large $n$, this will be true. This is because $\sum_{\theta}2^{-K^{\varphi_i}(\theta)}\leq 1$ for any $i$, and the tail sum approaches zero. If there are infinitely many $n$ through which $\alpha$ can be satisfied, there must exist a pair of large enough $n$ and a large enough stage so that $v_{\alpha,i}\leq 2^{-2e-i-2c_{\alpha}-3}$ for all $i\leq e$, and the relevant overlapping intervals have been moved in. In that case, we can let $A_s=(A_{s-1}\upharpoonright n)^{\smallfrown}\sigma$ as desired. So, $\alpha$ waits for an $n$ that makes sure we can compress the strings cheaply enough as described above. If such an $n$ is found, we act by extending $A_{s-1}\upharpoonright n$ and define $l_{\alpha}'=n+|\sigma|$ as well. Otherwise, let $A_{s} = A_{s-1}$.
\par
\emph{True outcomes} of the $G_e$-strategy: $w$ when $A$ avoids $S_e$; $s$ when $A$ meets $S_e$.
\par
We let the \emph{current outcome} of a $G_e$-strategy $\alpha$ be $w$ when it is still waiting for a pair of large enough stage and $n$ to act or $\alpha$ cannot be satisfied through any $n$. We let the current outcome be $s$ after it has acted by extending an initial segment of $A$ at a suitable stage. 
\vspace{\baselineskip}
\par
\emph{Construction}: 
Let a strategy $\alpha$ of length $t$ be eligible to act at a substage $t$ (we also say node $\alpha$ is visited) of stage $s\geq t$ if and only if $\alpha$ has the correct guess about the current outcomes of all $\beta\prec\alpha$. The strategy $\alpha$ will then act according to the above description. We define the current true path $f_s$ at stage $s$ to be the longest strategy eligible to act at stage $s$.
\vspace{\baselineskip}
\par
\emph{Verification}: Let $f=\liminf_s f_s$. First, we prove some combinatorial facts about our $\Pi_2$-arguments and strategies. This is what Lempp \cite{Tree} called a $\Pi_2$-Lemma.
\begin{lem}\label{pi2}
The following are true:
    \begin{enumerate}
        \item A true finitary outcome of a $D_i$-strategy or a true outcome $w$ or $s$ of a $G_e$-strategy is the current outcome at cofinitely many stages, and this outcome, once current, must be current from then on.
        \item A true infinitary outcome of a $D_i$-strategy is the current outcome at infinitely many stages. In the case of a true infinitary outcome, any current finitary outcome, once no longer current, can never be current again.
        \item If $\alpha\preceq f_{s_0}$ (for some stage $s_0$) and $\alpha\prec f$, then $\alpha$ is to the left of or above $f_s$ for all stages $s\geq s_0$
    \end{enumerate}
\end{lem}
\begin{proof}
    (1) and (2) are clear by the description of the $D_i$, $G_e$-strategies above.
    \par
    For (3), suppose a strategy with current outcome $o$, where $\beta^{\smallfrown}o\preceq\alpha$ for some $\beta$, changes after stage $s_0$. If $o$ is $w$ or $w_i$ for some $i$, then by (1), $\alpha$ can never be on the current true path again. On the other hand, if $o$ is $\infty$ or $s$, then this is clear by our ordering of $\Lambda$ and description of outcomes $w$ and $s$ above.
\end{proof}
\par
Notice that the lemma implies that $f$ is the true path. Now we claim that every strategy $\alpha\prec f$ ensures the satisfaction of its requirement.
\par
First, we show that $A$ exists.
\begin{lem}\label{interval}
Any interval assigned to some $\varphi_i$ would not be moved in infinitely often.
\end{lem}
\begin{proof}
We prove this by induction. The case of the first interval $\lb 0\rb$ is trivial since no $D_i$ nodes can mark it fresh. Suppose every interval before $I$ would not be moved in infinitely often. Then, $I$ could only be marked fresh finitely many times by the induction hypothesis. Therefore, it could only be moved in finitely many times.
\end{proof}
 
\begin{lem}\label{Delta2}
$\lim_{s\ra\infty}A_s(x)=A(x)$. 
\end{lem}
\begin{proof}
For any $x$, $x$ belongs to some interval $I$ assigned to some $\varphi_i$. By the previous lemma, $I$ would not be moved in infinitely often. Also, because $l_{\alpha}$ for a $G_e$ node is larger than $2^{2^{e+1}}$, nodes corresponding to only finitely many $G_e$ could do extensions involving $x$. For each $e$, since $l_{\alpha}$ is larger than $2^{2^{c_{\alpha}+1}}$ for a $G_e$ node $\alpha$, only finitely many such $\alpha$ can make extensions involving $x$, and each $\alpha$ could only act once. Therefore, $x$ could only be in the extensions of finitely many $G_e$ nodes. There must be a stage where all these extensions are completed. Therefore, after the last time $I$ is moved in and the completion of all such extensions in some stage $s$, $A_{s_0}(x)$ will be the same for any $s_0\geq s$.
\end{proof}
So, this is a $\Delta^0_2$-construction by Shoenfield limit lemma since $A_s(x)$ is computable. Now, we distinguish cases for $\alpha$.
\vspace{\baselineskip}
\par
Case 1: $\alpha$ is a $G_e$-strategy node. 
\begin{lem}\label{injury}
The $G_e$ node $\alpha$ on the eventual true path is not injured.
\end{lem}
\begin{proof}
There are threats from above, left, right, and below that could injure $G_e$. For threats from above, a $D_i$ node $\beta\prec\alpha$ with $i\leq e$ might injure $\alpha$ at later stages after $\alpha$ acts. But $\alpha$ only extends $A_{s-1}\upharpoonright n$ to $(A_{s-1}\upharpoonright n)^{\smallfrown}\sigma$ at stage $s$ when any interval assigned to $\varphi_i$ overlapping $\sigma$ has been moved in by $\beta$ with an infinite current outcome and will not be marked fresh again. The reason for the latter is: higher priority $D_i$'s nodes with infinite current outcomes cannot mark such intervals fresh anymore because $\alpha$ could only be satisfied when all of them have been moved in, and nodes with finite outcomes on the true path cannot mark any such interval fresh after stage $s$ because they cannot act anymore. For threats from the left, nodes to the left of $\alpha$ cannot injure $\alpha$ after its action because such nodes will not be visited by the third part of the $\Pi_2$-Lemma \ref{pi2}. For the threats from the right or below, that is, a higher priority $D_i$ node might have a wrong current outcome at some stage and a correct outcome at a later stage by correcting its guess about whether $\varphi_i$ is total or not, or a node $\beta\succ\alpha$ might injure $\alpha$ at or after the stage $\alpha$ acts. By the $\Pi_2$-lemma, these nodes are visited for the first time after $\alpha$ acts and will be initialized. Notice that $l_{\beta}$ for any newly initialized node $\beta$ is larger than $l_{\alpha}'$, and any $\sigma$'s in new extensions or intervals moved in are beyond the extension in $\alpha$'s action.
\end{proof}
\begin{lem}
A is 1-generic.
\end{lem}
\begin{proof}
Assume by induction that there is a stage $s$ after which we never act for any $G_{e'}$ node on the true path with $e' < e$. If we ever act for the $G_e$ node $\alpha$ on the true path as described in the $G_e$-strategy, we permanently satisfy requirement $G_e$ by ensuring that $A$ meets $S_e$ by the lemma above. In that case, $G_e$-strategy never acts again, so the induction can proceed. If not, there are only finitely many $n$ such that $A_{t-1}\upharpoonright n$ has an extension in $S_e[t]$ for all $t>s$ and no such $n$ enables the action of the $\alpha$. No sufficiently long initial segment of $A$ has an extension in $S_e$, so $A$ avoids $S_e$. The induction can proceed as well since $\alpha$ will not act at any stage.
\end{proof} 
\vspace{\baselineskip}
\par
Case 2: $\alpha$ is a $D_i$-strategy.
\begin{lem}\label{lowerbound}
Suppose that $\varphi_i$ is an order function. There are constants $c_i$ and $d_i$ such that the following holds. Suppose $\varphi_i$ is assigned to the interval $I_j$ and we move in $I_j$ at stage $s$ the last time without being marked fresh again. Then if $j$ is sufficiently large, no $G_e$ node makes an extension that overlaps $I_j$ after stage $s$, and $m$ is such that $\max I_j<m\leq \max I_{j+2^{i+1}}$, then
$K^{\varphi_i}(A\upharpoonright m+1)\geq c_im-d_i$.
\end{lem}
\begin{proof}
We first define a machine $M$. On input $\rho$, run the optimal prefix-free machine $\mathbb{U}$ on $\rho$. If $\mathbb{U}$ outputs $\gamma$ and $\gamma$ has length $\max I_j+1$ for some $j$, let $M$ output the string $\delta$ such that $(\gamma\upharpoonright \min I_j)^{\smallfrown}\delta=\gamma$. By Remark \ref{g0}, $\mathbb{U}$ can simulate the machine $M$ on input $\rho$ in time $g_M(g_N(t))$, for some computable function $g_M$, if $\mathbb{U}$ runs in time $g_N(t)$ when $M$ outputs $\gamma$. Let $d_M$ be the coding constant for the machine $M$. Let $b =\varphi_i(\max I_{j+2^{i+1}}+1)$. 
\par
We claim that if $j$ is large enough, then $K_{g_N(b)}(A\upharpoonright \max I_j+1) >\frac{\max I_j}{3}$. We can let $j$ be large enough such that $\frac{\max I_j}{3}+d_M\leq |I_j|-1$. Let $s$ be the stage that we move in $I_j$ the last time without being marked fresh again. Suppose $\tau$ is the string such that $A_s\upharpoonright \max I_j+1=(A_{s-1}\upharpoonright \min I_j)^{\smallfrown}\tau$ at stage $s$. Then, we have $K_{g_M(g_N(b))}(\tau)\geq |I_j|$ because $\tau$ is not an output of the optimal prefix-free machine with input length smaller than $|I_j|$ and time $g_M(g_N(b))$. Suppose that $K_{g_N(b)}(A\upharpoonright \max I_j+1) \leq \frac{\max I_j}{3}$. Then, there exists a string $\rho$ of length $K_{g_N(b)}(A\upharpoonright \max I_j+1)\leq \frac{\max I_j}{3}$ such that $\mathbb{U}_{g_N(b)}(\rho)=A\upharpoonright \max I_j+1$. Since $\mathbb{U}$ can simulate the machine $M$ in time $g_M(g_N(t))$, there is a string $\rho_0$ of length at most $d_M+K_{g_N(b)}(A\upharpoonright \max I_j+1)$ such that $\mathbb{U}_{g_M(g_N(b))}(\rho_0)=\tau$. But then $|I_j|\leq K_{g_M(g_N(b))}(\tau)\leq d_M+K_{g_N(b)}(A\upharpoonright \max I_j+1)\leq \frac{\max I_j}{3}+d_M\leq |I_j|-1$ since $j$ is large enough. Contradiction. Therefore, the claim holds.
\par
We claim that there is $c_i$ such that for $j$ and $m$ as in the statement, $c_im\leq \frac{\max I_j}{4}$. Note that $\max I_j=2^{j+1}-2$ and $\varphi_i$ is assigned to every $2^{i+1}$th interval, we can choose $c_i=2^{-2^{i+1}-3}$.
\par
Now we define another machine $N=N_i$ as follows. On input $\rho$, run the universal prefix-free machine $\mathbb{U}$ on $\rho$. If $\mathbb{U}$ outputs $\gamma$ and $\max I_k+1<|\gamma|\leq\max I_{k+2^{i+1}}+1$ where the interval $I_k$ is assigned to $\varphi_i$, let $N$ outputs the string $\gamma\upharpoonright \max I_k+1$. Therefore, $\mathbb{U}$ can simulate $N$ on input $\rho$ in time $g_{N}(t)$, for some computable function $g_N$, if $\mathbb{U}$ runs in time $t$ when $N$ outputs $\gamma\upharpoonright \max I_k+1$. Let $d_{N}$ be the coding constant for the machine $N$.
\par
Let $j$ be large enough. Suppose that $K^{\varphi_i}(A\upharpoonright m+1)<c_im -d_N$. Because $b\geq \varphi_i(m+1)$, we have $K_b(A\upharpoonright m+1)\leq K^{\varphi_i}(A\upharpoonright m+1)<c_im-d_N$. Then, there exists a string $\rho$ of length smaller than $c_im-d_N$ such that $\mathbb{U}_b(\rho) =A\upharpoonright m+1$. Because $\mathbb{U}$ can simulate the machine $N$ in time $g_N(t)$, there is a string $\rho_0$ of length smaller than $c_im-d_N+d_N$ such that $\mathbb{U}_{g_N(b)}(\rho_0)=A\upharpoonright \max I_j+1$. But then $K_{g_N(b)}(A\upharpoonright \max I_j+1)<c_im-d_N+d_N=c_im\leq \frac{\max I_j}{4} <\frac{\max I_j}{3}$, contradicting the previous claim. So we can take $d_i=d_N$ to satisfy the lemma.
\end{proof}
\begin{lem}\label{upperbound}
$K(A\upharpoonright m)$ has an upper bound $O(\log(m))$.
\end{lem}
\begin{proof}
It suffices to show that $C(A\upharpoonright m)$ is bounded by $O(\log(m))$ since $K(x)\leq^+C(x)+2\log(|x|)$ for any $x$. Here, $\log n=\max\lb k\in\N:2^k\leq n\rb$. First, note that very few $G_e$ nodes can affect $A\upharpoonright m$. By definition, $l_{\alpha}>2^{2^{e+1}}$ for any $G_e$ node, so $e<\log(\log(m))$. For each such $e$, since $l_{\alpha}>2^{2^{c_{\alpha}+1}}$ for any $G_e$ node $\alpha$, there are at most $\log(\log(m))$ many nodes that can affect $A\upharpoonright m$. Also, there are $\log(m)+1$ many intervals before or containing $m-1$. Therefore, we can code whether each of such intervals is moved in without being marked fresh again, and whether each of the first $\log(\log(m))$ $G_e$ nodes visited for each $e< \log(\log(m))$ acts or not with a string of length $L_m=\log(m)+1+(\log(\log(m)))^2$. Note that $L_m<2\log(m)$ when $m>2^{2^5}$. We can add zeros at the end of the string to make a string of length $2\log(m)$.
Now we define a machine $M$ that has input $\sigma\tau$ where $|\sigma|=2\log(m)$, $|\tau|=\log(m)+1$, and $m>2^{2^5}$. Also, $\tau$ is the binary string for number $m$. The machine $M$ can acquire $\sigma$ and $\tau$ from the length $|\sigma\tau|$. Next, it will simulate the construction of $A$ until the current outcomes of $G_e$ nodes and intervals moved in coincides with the information encoded in $\sigma$. Then, $M$ outputs the first $m$ bits of the string it has constructed. There exists a $\sigma$ such that $M$'s output is $A\upharpoonright m$ with such $\sigma\tau$ as input because any $G_e$ node could only act once and an interval cannot be marked fresh unless an interval before it is moved in. Therefore, $C(A\upharpoonright m)$ is bounded by $O(\log(m))$ for any $m$. 
\end{proof}
\begin{lem}
$A$ is deep.
\end{lem}
\begin{proof}
We are going to show that each $D_i$ requirement is met. Assume $\varphi_i$ is an order function. Let $\gamma$ be the $D_i$ node on the eventual true path. For any $c$, by Lemma \ref{lowerbound} and \ref{upperbound}, there is a large enough $n_0$ such that, if $m\geq n_0$ and no $G_e$ strategies changed the string in the interval $I_j$ that is assigned to $\varphi_i$ and $\max I_j<m\leq \max I_{j+2^{i+1}}$, $K^{\varphi_i}(A\upharpoonright m+1)>K(A\upharpoonright m+1)+c$.
\par
Let $d_0$ be the coding constant corresponding to the machine obtained by the bounded request set $R$ defined in the $G_e$-strategy. Notice there are only finitely many $\alpha$ such that $e+c_{\alpha}\leq c+d_0$. So, there is a large enough $n_1$ such that $n_1$ is beyond the concatenating segment of the extension made by any such $\alpha$ that overlaps a moved-in $\phi_i$ interval.
\par
Next, we consider the case that a $G_e$ node $\alpha$ with $e+c_{\alpha}>c+d_0$ extends an initial segment of $A_{s-1}$ and changes the string in the $I_j$ interval, which is assigned to $\varphi_i$, moved in, and beyond $l_{\gamma}$. When $i>e$, there are three cases to consider: $\alpha\prec\gamma$, $\alpha$ is to the left of $\gamma$, and $\alpha$ is to the right of $\gamma$. We do not worry about the first two cases since $\gamma$ would be initialized by Lemma \ref{pi2} and $l_{\gamma}$ would be beyond the extension. For the third case, there is a node $\tau$ that is the longest shared prefix of $\gamma$ and $\alpha$. Node $\tau$ is either a $G_{e'}$ node with $e'<e$ that makes an extension, or a $D_t$ node with $t\leq e$ that realizes it made a wrong guess. For the former, $l_{\gamma}$ is beyond the extension $\alpha$ made as well because $\gamma$ will be initialized after $\tau$'s action. For the latter, $I_j$ will be marked fresh and moved in again because there is one interval before $l_{\alpha}$ assigned to $\phi_t$ not moved in when $\alpha$ is initialized. However, there is a risk that this interval will be moved in by a $D_t$ node $\eta$ before $\alpha$ acts. Suppose such an $\eta$ exists. Notice that $\gamma$ is below $\tau$'s infinite outcome while $\alpha$ is below one of $\tau$'s finite outcomes since $\gamma$ is on the true path and $\alpha$ is to the right of it. So, $\tau$ cannot be this $\eta$ because $\alpha$ cannot be visited again to act after the interval is moved in by Lemma \ref{pi2}. Therefore, $\alpha,\eta$ are not comparable and one of them is to the right of the other. Then, both $\alpha$ and $\eta$ would be visited twice as the following must happen in sequence: $\eta$ initializes, $\alpha$ initializes, $\eta$ moves in the interval, and $\alpha$ acts. There is no risks when $\alpha$ initializes before $\eta$ initializes because in this case $l_{\eta}$ is beyond $l_{\alpha}$, which is beyond that interval assigned to $\phi_t$. Let $\beta$ be the longest shared initial segment of $\alpha,\eta$. It is impossible for $\beta$ to visit outcomes $w,s,w$, or $w_x,w_y,w_x$, or $w_x,\infty,w_x$, where $w_y$ is to the left of $w_x$. So no matter which of $\alpha$ or $\eta$ is on the right, the above sequence is impossible. Contradiction. So, the risk is cleared. When $i\leq e$, we utilize the compressions we have made. We know that the weight of the request set $R$ is no larger than $1$ since 
\begin{align*}
    \sum_{e=0}^{\infty}\sum_{\alpha\in P_e}\sum_{i=0}^{e}\sum_{\theta\in N_{\alpha,i}}2^{e+c_{\alpha}-K^{\varphi_i}(\theta)} & \leq\sum_{e=0}^{\infty}\sum_{\alpha\in P_e}\sum_{i=0}^{e}2^{e+c_{\alpha}}2^{-2e-i-3-2c_{\alpha}}\\ & =\sum_{e=0}^{\infty}\sum_{\alpha\in P_e}\sum_{i=0}^{e}2^{-e-i-3-c_{\alpha}}\leq 1,
\end{align*}
where $P_e$ is the set of all $G_e$ nodes visited.
Then, $K^{\varphi_i}(A\upharpoonright m+1)\geq K(A\upharpoonright m+1)+e+c_{\alpha}-d_0> K(A\upharpoonright m+1)+c$ will hold for $\max I_j<m\leq \max I_{j+2^{i+1}}$. By our definition of $N_{\alpha,i}$, there are requests in $R$ ensuring this since the $\theta$'s in $N_{\alpha,i}$ can be any string of length between $\max I_j$ and $\max I_{j+2^{i+1}}$. 
\par
Therefore, for any $i$ such that $\varphi_i$ is an order function, for any $c$, there is an $n>n_0,n_1,l_{\gamma}$ such that $K^{\varphi_i}(A\upharpoonright m+1)>K(A\upharpoonright m+1)+c$ for all $m\geq n$.
\end{proof}

%
%

\bibliographystyle{plainurl}
\bibliography{references} 

\begin{thebibliography}{1}

\bibitem{Deep}
Charles~H Bennett.
\newblock {\em Logical depth and physical complexity}.
\newblock Citeseer, 1988.

\bibitem{downey}
R.G. Downey and D.R. Hirschfeldt.
\newblock {\em Algorithmic Randomness and Complexity}.
\newblock Theory and Applications of Computability. Springer New York, 2010.

\bibitem{LowDeep}
Rod Downey, Michael McInerney, and Keng~Meng Ng.
\newblock Lowness and logical depth.
\newblock {\em Theoretical Computer Science}, 702:23--33, 2017.

\bibitem{TimeBound}
Rupert H{\"o}lzl, Thorsten Kr{\"a}ling, and Wolfgang Merkle.
\newblock Time-bounded {K}olmogorov complexity and {S}olovay functions.
\newblock {\em Theory of Computing Systems}, 52(1):80--94, 2013.

\bibitem{poly}
David~W Juedes and Jack~H Lutz.
\newblock Modeling time-bounded prefix {K}olmogorov complexity.
\newblock {\em Theory of Computing Systems}, 33:111--123, 2000.

\bibitem{Tree}
Steffen Lempp.
\newblock Priority argument in computability theory, model theory, and complexity theory.
\newblock URL: \url{https://people.math.wisc.edu/~lempp/papers/prio.pdf}.

\bibitem{LiVitanyi}
Ming Li, Paul Vit{\'a}nyi, et~al.
\newblock {\em An introduction to {K}olmogorov complexity and its applications}, volume~3.
\newblock Springer, 2008.

\bibitem{Nies}
Andr{\'e} Nies.
\newblock {\em Computability and randomness}, volume~51.
\newblock OUP Oxford, 2012.

\end{thebibliography}

\end{document}